\title{Thick groups have trivial Floyd boundary}
\author{Ivan Levcovitz}
\date{}

\documentclass[11pt]{article}
\usepackage{amsmath}
\usepackage{amssymb}
\usepackage{amsthm}

\usepackage{hyperref}

\usepackage{graphicx}
\usepackage{overpic}
\usepackage{enumitem}

\newtheorem{theorem}{Theorem}[section]
\newtheorem{lemma}[theorem]{Lemma}
\newtheorem{proposition}[theorem]{Proposition}

\newtheorem{remark}[theorem]{Remark}

\newtheorem{introquestion}{Question}
\newtheorem*{maintheorem}{Main Theorem}

\theoremstyle{definition}
\newtheorem{definition}[theorem]{Definition}

\usepackage[]{geometry}

\usepackage{sectsty}

\sectionfont{\fontsize{12}{15}\selectfont \centering}

\subsectionfont{\fontsize{12}{15}\selectfont }

\begin{document}
\maketitle

\begin{abstract}
	 We prove that thick groups (and more generally thick graphs) have trivial Floyd boundary. This shows a wide class of finitely generated groups that are non-relatively hyperbolic have trivial Floyd boundary. In addition to giving new examples, our result provides a common proof and framework for many of the known results in the literature.
\end{abstract}

Floyd introduced one of the first compactifications of an arbitrary finitely generated group \cite{Floyd}, now known as the Floyd boundary. The Floyd boundary in general depends on a choice of a scaling function known as a Floyd function. This compactification turns out to be strongly linked to the theory of relatively hyperbolic groups, as a crucial result of Gerasimov shows there is a continuous equivariant map from the Floyd boundary of a relatively hyperbolic group to the group's Bowditch boundary \cite{Gerasimov}.
 
Gerasimov's result is the culmination of a series of previous results showing there are continuous equivariant maps from the Floyd boundary to other commonly studied ``hyperbolic boundaries'' of groups. Given a geometrically finite discrete subgroup of Isom($\mathbb{H}^3$) (the isometry group of hyperbolic $3$--space), Floyd shows there is a continuous equivariant map from the Floyd boundary to the group's limit set when the Floyd function is taken to be $\frac{1}{n^2+1}$ \cite{Floyd}. Tukia generalizes Floyd's result to geometrically finite discrete subgroups of Isom($\mathbb{H}^m$) for integers $m > 3$ \cite{Tukia}. Given a hyperbolic group, Gromov shows its Floyd boundary is homeomorphic to the Gromov boundary \cite{Gromov_hyperbolic}. Both Gromov's and Gerasimov's result use Floyd function $\lambda^n$, for some $0 < \lambda < 1$. Roughly speaking, more rapidly decreasing Floyd functions detect less points of the boundary. Gerasimov's result, in particular, strengthens Tukia's and Floyd's results on Kleinian groups. 

We remark that the Floyd boundary has been shown to have connections to many other subject areas such as convergence actions \cite{K}, random walks on groups \cite{GGPW}, asymptotic cones \cite{OOS} and acylindrical hyperbolicity \cite{Sun, Wen-yuan}. Additionally, Gerasimov's map is further studied in \cite{GP}.

A consequence of Gerasimov's result is that the Floyd boundary of a relatively hyperbolic group is non-trivial. On the other hand, there do not seem to be any known examples of groups that are non-relatively hyperbolic and have non-trivial Floyd boundary:

\begin{introquestion}
	Is the Floyd boundary trivial for every finitely generated, non-relatively hyperbolic group?
\end{introquestion}

A positive answer to the above question, together with Gerasimov's result, would show that the existence of a non-trivial Floyd boundary characterizes relatively hyperbolic groups. We note that the above question was also asked in \cite[Problem 7.11]{OOS}.

Thick groups (and more generally thick metric spaces) were introduced in \cite{BDM}, and these authors show thick groups are non-relatively hyperbolic. We give an affirmative answer to the above question for the class of thick groups and more generally thick graphs:

\begin{maintheorem}
	Let $X$ be a thick graph, then the Floyd boundary of $X$ is one point with respect to any Floyd function. In particular, thick groups have trivial Floyd boundary.
\end{maintheorem}

We note that we use a slightly strengthened version of thickness than the original definition. However, our definition is still weaker than the one given in \cite{BD}. We refer the reader to Section \ref{sec_thick_overview} for the definition of thickness and to Section \ref{sec_main_thm} for the definition of a thick graph.

Many commonly studied non-relatively hyperbolic groups are known to be thick in all mentioned definitions and by the above theorem have trivial Floyd boundary. We list some of these examples:
\begin{enumerate}
	\itemsep0em 
	\item \label{ex_mcg} Mapping class groups of surfaces satisfying $3g + n -3 >1$ where $g$ is the genus and $n$ is the number of boundary components \cite{BDM} \cite{Beh} 
	\item \label{ex_cox} Coxeter groups that are non-relatively hyperbolic \cite{BHSC}
	\item \label{ex_artin} Artin groups that are non-relatively hyperbolic (equivalently those with connected defining graph) \cite{BDM} \cite{CP}
	\item \label{ex_out} Out($F_n$) and Aut($F_n$) for $n \ge 3$ \cite{BDM}
	\item \label{ex_3M} Fundamental group of a non-geometric graph $3$--manifold \cite{BD}
	\item \label{ex_free_by_cyclic} Non-relatively hyperbolic free-by-cyclic groups \cite{Hag3}
	\item \label{ex_product} The product of two infinite groups \cite{BDM}
	\item \label{ex_law} Groups satisfying a law, such as Solvable groups and Burnside groups \cite{BDM}
	\item \label{ex_central} Groups with a central element of infinite order \cite{BDM}
	\item \label{ex_graphs} Graphs of groups with infinite edge groups and with vertex groups thick of order at most $n$ \cite{BD}
	\item \label{ex_teich} Teichm\"{u}ller space with the Weil-Peterson metric for surfaces of type $3g + n - 3 \ge 6$ (with $g$ and $n$ as in \ref{ex_mcg}) is quasi-isometric to a thick graph \cite{BDM}
\end{enumerate}

Many of the above groups were previously known to have trivial Floyd boundary, and we review some of these results. Floyd shows both that the product of two infinite groups and Nilpotent groups have trivial boundary \cite{Floyd}. Karlsson proves that if a group does not contain a non-abelian free group of rank $2$, then its Floyd boundary is trivial \cite{K}. Karlsson-Noskov give conditions on a group's generating set that imply trivial Floyd boundary \cite{KN}. In particular, it can be deduced from this last result that Artin groups with connected graph, Aut($F_n$) (for $n \ge 5$) and the mapping class groups listed above have trivial Floyd boundary. Our main theorem unifies many of these known results. Furthermore, our argument does not rely on a group action and uses only the underlying metric structure of the group's Cayley graph. As far as we are aware, it is a new result that the Floyd boundary is one point for the groups/spaces listed in  \ref{ex_cox},  \ref{ex_out} (for the cases not mentioned above as known), \ref{ex_3M}, \ref{ex_free_by_cyclic}, \ref{ex_graphs} and \ref{ex_teich}.

The results from \cite{K} imply that a finitely generated group admits a convergence action on its Floyd boundary. Consequently, if a finitely generated group does not admit any non-trivial convergence action, then its Floyd boundary must be trivial. An answer to the following question would be very interesting as it would provide us with a better understanding of the connection between convergence actions, thickness (strictly a metric property) and the Floyd boundary:

\begin{introquestion} \label{intro_q_2}
	Is there a thick group that admits a non-trivial convergence action?
\end{introquestion}

\subsubsection*{Acknowledgments:} I thank my doctoral advisor Jason Behrstock for his guidance while much of this was written. I am also thankful to Ilya Gekhtman and Victor Gerasimov for useful conversations regarding the Floyd boundary. I would also like to thank the anonymous referees for their excellent comments and suggestions. A hypothesis on the Floyd function in the main theorem of an earlier draft was dropped thanks to a referee's observation. Additionally, the important connection to convergence actions, including Question \ref{intro_q_2} of the introduction, was pointed out by both a referee and Gerasimov.

\section{Preliminaries} \label{sec_prelims}
Let $(X, d)$ be a metric space. We let $B_x(r)$ denote the ball of radius $r \ge 0$ centered at a point $x \in X$. Given a subspace $Y \subset X$, we let $N_r(Y)$ denote the $r$--neighborhood of $Y$.

\subsection{Floyd Boundary} 
In the definition of a Floyd boundary, there is some choice in which scaling functions are permissible. We follow the definition from \cite{GP}.

Let $X$ be a locally finite connected graph endowed with a basepoint $b \in X$. For instance, one can take $X$ to be the Cayley graph of a finitely generated group and $b$ the identity element. Let $d(~,~)$ be the path metric on this graph where each edge is assigned length $1$. Given a path $p$ in $X$, we let $|p|$ denote its length.

Let $f: \mathbb{Z}^{+} \to \mathbb{R}^{+}$ be a function satisfying the following two conditions: 
\begin{enumerate}[label=\alph*)]
	\item \label{floyd_func_cond_a} $\exists K \ge 1$ such that $\forall n \in \mathbb{Z}^{+}$ : $1 \le \frac{f(n)}{f(n+1)} \le K$
	
	\item \label{floyd_func_cond_b} $\Sigma_{n = 1}^{\infty}{f(n)} < \infty$ 
\end{enumerate}

We call $f$ a \textit{Floyd function}. To simplify the construction of the Floyd boundary, for any Floyd function, $f$, we define $f(0) := f(1)$. For most purposes it is sufficient to consider the Floyd function $f(n) = \frac{1}{n^2}$.

We construct a new metric space, $X_f$, by assigning a length to each edge of $X$ that depends on the edge's distance from $b$. As graphs (without a metric) $X_f$ and $X$ are isomorphic. The length of an edge $e \in X_f$ between vertices $\{v_1, v_2\}$ is $f(n)$, where $n = d(b, e) = d(b,\{v_1,v_2\})$. We call this the \textit{Floyd length} of $e$.

If $p$ is a path in $X_f$, given by consecutive edges $e_1, e_2, ..., e_n$, its \textit{Floyd length}, $|p|_f$, is the sum the of the Floyd lengths of the edges, i.e. $|p|_f = \Sigma_{i=1}^{n}f(d(e_i, b))$. Given vertices $u, v \in X_f$, their \textit{Floyd distance}, $d_f(u, v)$, is the infimum of the Floyd lengths of paths from $u$ to $v$.

The Cauchy completion $\bar{X}_f$ of the metric space $(X_f, d_f)$ is called the \textit{Floyd completion}. The subspace $\partial_f X  = \bar{X}_f \setminus X$ is the \textit{Floyd boundary}.

We say a Floyd boundary is trivial if it consists of either zero, one or two points. Let $g$ and $f$ be Floyd functions such that $g \le f$. By an easy argument, if $\partial_f X$ is trivial, then $\partial_g X$ is trivial as well.

Given a finitely generated group, its Floyd boundary with respect to a Floyd function $f$ is the Floyd boundary of the group's Cayley graph with respect to $f$. The Floyd boundary, up to Lipschitz equivalence, of a finitely generated group does not depend on a choice of finite generating set when Floyd function $f = \frac{1}{n^2}$ is used \cite{Floyd}.

A path, $p: I \to X$, in a metric space $(X, d)$ is a $C$--\textit{quasi-geodesic} if given any vertices $v, v'$ in the image of $p$, we have the inequalities:
\[\frac{1}{C}d(v,v') - C \le |p(v,v')| \le Cd(v,v') + C\]
where $p(v,v')$ is the subpath of $p$ from $v$ to $v'$. 

We say a graph is \textit{locally finite} if each edge has finite valence. Throughout this paper, we will make use of the Karlsson Lemma, first proved in \cite{K} and generalized in \cite{GP}, which shows quasi-geodesics lying outside a large set are small in the Floyd metric.

\begin{lemma}[Karlsson Lemma, {\cite[Lemma 2.2]{GP}}] \label{lemma_karlsson}
	Given a locally finite graph $X$, Floyd function $f$ and constants $\epsilon, ~ C>0$, there exists a finite set of vertices $K \subset X_f$, such that every $C$-quasi-geodesic which does not intersect $K$ has Floyd length less than $\epsilon$.
\end{lemma}

\subsection{Divergence}
One may roughly think of the divergence function of a metric space as the best upper bound on the rate a pair of geodesic rays can stray apart from one another. There are many definitions of divergence in the literature. We use here the definition as in \cite{DMS} and \cite{BD}. In section \ref{sec_divergence_thm}, we relate the divergence of a space to its Floyd boundary.

Let $X$ be a metric space. Fix $0< \delta < 1$ and $\gamma \ge 0$. For any three points $a,b,c \in X$ such that $d(c, \{a,b\}) = r > 0$, define:
\[ \text{div}_{\gamma, \delta}(a,b,c) =  \inf \{ |\alpha| \} \]
where $\alpha$ is a path connecting $a, b$ that does not intersect the ball $B_c(\delta r - \gamma)$, and $|\alpha|$ is this path's length. If no such paths exists, define $\text{div}_{\gamma, \delta}(a,b,c) = \infty$.

The \textit{divergence function} $\text{Div}(X) = \text{Div}_{X, \gamma, \delta}(n)$ of the space $X$ is defined as the supremum of $\text{div}_{\gamma, \delta}(a,b,c)$ taken over all $a, b, c$ such that $d(a,b) \le n.$ 

For a large class of metric spaces (including all finitely generated groups) the divergence function is a quasi-isometry invariant up to an equivalence relation on functions \cite{DMS}.

\subsection{Thick spaces} \label{sec_thick_overview}

This subsection gives an overview of the definition of a thick space. We work with the original definition of thickness from \cite{BDM}, with the extra assumption that thick of order 0 spaces are wide. This assumption is also made in the ``strong'' definition from \cite{BD}; however, we do not require the full strength of the definition in \cite{BD}. We point out these differences when relevant.  We refer the reader to the mentioned references for further background on thick spaces.

We first define wide spaces, the elementary building blocks of a thick space.

\begin{definition}[Wide Space] \label{def_wide} A metric space, $X$, is \textit{$C$--wide} if: 
	\begin{enumerate}
		\item Any $x \in X$ is in the $C$ neighborhood of some bi-infinite $C$--quasi-geodesic. 
		\item There exist constants $0 < \delta < 1$ and $\gamma \ge 0$, such that the divergence of $X$, $\text{Div}^{X}_{\gamma, \delta}(n)$, is bounded above by a linear function.
	\end{enumerate}
\end{definition}

\begin{remark}
	The definition given above is slightly different than the usual definition given in terms of asymptotic cones, as the above formulation is more convenient in our setting. However, by \cite[Proposition 1.1]{DMS} the definition above is equivalent to the usual one when $X$ is the Cayley graph of a finitely generated group (and for many other general metric spaces).
\end{remark}

Roughly, $X$ is thick of order $k$ if it is the coarse union of subspaces that are each thick of order at most $k-1$. Furthermore, any two of these subspaces can be ``thickly'' connected by a sequence of these subspaces. This is formally defined below.
\begin{definition}[Thick Space] \label{def_thick} A metric space is \textit{$C$--thick of order 0} if it is $C$--wide.
	
	We say that a metric space $X$ is \textit{ $C$--thick of order at most $k$ with respect to a collection of subsets $\mathcal{Y} = \{ Y_{\alpha} \}$ } if 
	
	\begin{enumerate}
		\item $X = \bigcup_{\alpha \in  \mathcal{A}}{N_C(Y_{\alpha})}$, i.e. $\mathcal{Y}$ coarsely covers $X$.
		\item Every $Y \in \mathcal{Y}$ with the induced metric is $C$--thick of order at most $k-1$.
		\item For every $Y,~ Y' \in \mathcal{Y}$, there exists a sequence of subspaces in $\mathcal{Y}$:
		\[ Y=Y_1, ~Y_2,...,~Y_{n-1},~Y_m = Y'\] 
		 such that $N_C(Y_i) \cap Y_{i+1}$ has infinite diameter, for $1 \le i <m$.
	\end{enumerate}
	
	$X$ is \textit{thick of order k} if $X$ is $C$--thick of order at most $k$ for some $C >0$, and $X$ is not $C'$--thick of order at most $k-1$ for any $C'>0$. 
\end{definition}

\begin{remark}
	The above definition is weaker than that of \cite{BD} in two ways. Firstly, the wide subspaces in a thick structure are not required to have divergence uniformly bounded by the same linear function. Furthermore, the infinite diameter intersections in the definition are not required to be coarsely path connected.
\end{remark}

We say a finitely generated group $G$ is a \textit{thick group} if the Cayley graph of $G$ with respect to a (equivalently, any) finite generating set is a thick space. We remark that this is strictly a metric property and is weaker than the property of algebraically thickness from \cite{BDM}. 

The following lemma is a straightforward consequence of the definitions.

\begin{lemma} \label{lemma_bi_inf_geos}
	Let $X$ be $C$--thick of order at most $k$ with respect to the collection of subsets $\mathcal{Y}$. Given any $Y \in \mathcal{Y}$ and $y \in Y$, $y$ is in the $(k+1)C$ neighborhood of some bi-infinite $C$--quasi-geodesic contained in $Y$. 
\end{lemma}
\begin{proof}
	If $X$ is thick of order $0$, then the claim is satisfied by the definition of a wide space (Definition \ref{def_wide}). Otherwise, if $X$ is thick of order at most $k > 0$, then by Definition \ref{def_thick}, every point of $X$ is distance at most $C$ from a thick of order $k-1$ space. The claim then follows by induction.
\end{proof}

\section{Divergence and the Floyd boundary} \label{sec_divergence_thm}

If a Floyd function decays rapidly in comparison to the divergence function, then the Floyd boundary must be one point:

\begin{proposition} \label{prop_div_floyd}
Let $X$ be a locally finite, infinite, connected graph with divergence function $D(n) = Div_{X, \gamma, \delta}(n)$, and let $f$ be a Floyd function satisfying
\[\limsup_{n \to \infty}{ \Big( D(2n) \cdot f(\delta n- \gamma) \Big) } = 0\]
then the Floyd boundary, $\partial_f X$, is one point.
\end{proposition}

\begin{proof} Let $b$ be the basepoint used in constructing the Floyd boundary. We will prove that given $\epsilon > 0$, there exists an $N$ such that for all $x,y \in X$ with $d(x, b), d(y, b) >N$ we have $d_f(x,y) < \epsilon$ (recall $d_f$ is the Floyd distance). 

Choose $N$ such that for $n > N$, the following two conditions are satisfied:
\begin{enumerate}
\item \label{choices_item1} Any geodesic, $\beta$, in the $d(,)$ metric, that does not intersect the ball $B_{b}(\delta n - \gamma)$ has Floyd length $|\beta|_f < \frac{\epsilon}{2}$.

\item \label{choices_item2} $D(2n) \cdot f(\delta n- \gamma) < \frac{\epsilon}{2}$. 
\end{enumerate}
Such an $N$ exists satisfying (\ref{choices_item1}) by Lemma \ref{lemma_karlsson}. We can further choose $N$ large enough to satisfy (\ref{choices_item2}) by our assumption on $D(n)$.

Let $x,y \in X$ be such that $d(x, b) \ge d(y, b) >N$. Set $r = d(y, b)$. Fix the ball $B = B_{b}(\delta N - \gamma)$. Let $\beta$ be a geodesic from $x$ to $b$. Let $x'$ be the point on $\beta$ that is distance $r$ from $b$ and is closest to $x$. Let $\beta'$ be the segment of $\beta$ from $x$ to $x'$. Note that $\beta' \cap B = \emptyset$. Therefore, by condition \ref{choices_item1}, we have that $|\beta'|_f \le \frac{\epsilon}{2}$.

Let $\alpha$ be a shortest path from $x'$ to $y$ avoiding the ball $B$. As $d(x', y) \le 2N$ and $\alpha$ remains outside of $B$, we can guarantee that $|\alpha| \le D(2N)$. Every edge outside the ball $B$ has Floyd length at most $f(\delta n- \gamma)$. We get the following bound on the Floyd length of $\alpha$:
\[|\alpha|_f \le |\alpha| f(\delta n- \gamma) \le D(2n) \cdot f(\delta n- \gamma)\] 
Therefore, by condition \ref{choices_item2}, we have that $|\alpha|_f \le \frac{\epsilon}{2}$. The composition of $\beta'$ followed by $\alpha$ gives a path from $x$ to $y$ of Floyd length less than $\epsilon$. This proves the claim.
\end{proof}

Given a Floyd function $f$, it follows from condition \ref{floyd_func_cond_a} in the definition of a Floyd function that $\lfloor \frac{n}{2} \rfloor f(n) \le f(n) + f(n-1) + \dots + f( \lfloor \frac{n}{2} \rfloor)$ for any positive integer $n$. Furthermore, by condition \ref{floyd_func_cond_b} in the definition of a Floyd function, the right side of this equation must limit to $0$. Thus, any Floyd function grows sublinearly. We thus get the following consequence of the above proposition, which is used in the next section.

\begin{remark} \label{rmk_linear_div}
	If $X$ has linear divergence, then its Floyd boundary is one point for any Floyd function.
\end{remark}

\section{Proof of main theorem} \label{sec_main_thm}

Before proving the main theorem, we first prove the following lemma.

\begin{lemma} \label{lemma_inf_ray}
	Let $X$ be a metric space with the property that every point is distance at most $C$ from a bi-infinite $C$--quasi-geodesic. Fix $b \in X$. There exist constants $K \ge 2$ and $R \ge 0$, only depending on $C$, such that given any $r>R$ and $x \notin B = B_b(r)$, there exists an infinite $C$--quasi-geodesic ray, distance at most $C$ from $x$, that does not intersect the ball $B' = B_b(\frac{r}{KC})$.
\end{lemma}
\begin{proof}
	Fix a choice of $K$ and $r$. Let $x \in X \setminus B$. By assumption, there exists a bi-infinite $C$--quasi-geodesic, $\alpha$, that is $C$--close to $x$. Let $x' \in \alpha$ be such that $d(x, x') \le C$. Let $\alpha_1$ and $\alpha_2$ be the $C$--quasi-geodesic rays based at $x'$, obtained by following $\alpha$ in opposite directions towards infinity.
	
	Suppose $\alpha_1$ and $\alpha_2$ each intersect $B'$ at points $p_1$ and $p_2$ respectively. It follows that $d(p_1,p_2) < \frac{2r}{KC}$. Furthermore, as $d(x,p_1) > \frac{(KC-1)r}{KC}$ and $d(x,x') \le C$, we have that $d(x',p_1) > \frac{(KC-1)r}{KC} - C$. 
	
	Given a quasi-geodesic $\beta$ and points $x, y \in \beta$, we let $\beta(x,y)$ denote the subsegment of $\beta$ between $x$ and $y$. Using the established inequalities and the quasi-geodesic inequalities, we conclude the following:
	\[\frac{2r}{K} + C > Cd(p_1,p_2) + C \ge |\alpha(p_1,p_2)| \ge |\alpha_1(x',p_1)| \ge d(x',p_1) \ge \frac{(KC-1)r}{KC} - C\]
	
	However, there exist constants $K$ and $R$ such that for any $r>R$, we have that $\frac{2r}{K} + C <  \frac{(KC-1)r}{KC} - C$. These contradicting inequalities imply that either $\alpha_1$ or $\alpha_2$ does not intersect $B$ for such choices of $K$ and $R$. This proves the claim.
\end{proof}

For convenience, we name the class of graphs considered in the main theorem.

\begin{definition} [Thick graph]
	Suppose $X$ is an infinite, connected, locally finite graph where each edge is given length $1$. Suppose that $X$ is thick of order $k$ for some $k$ in the path metric. We call such a graph a \textit{thick graph}. In particular, the Cayley graph of a finitely generated thick group is a thick graph.
\end{definition}

\begin{theorem}
Let $X$ be a thick graph, and let $f$ be any Floyd function. Then the Floyd boundary $\partial_f X$ is one point.
\end{theorem}
\begin{proof}
Let $b \in X$ be the basepoint used in constructing the Floyd boundary. As usual, we denote by $d(,)$ the metric in $X$ and by $d_f(,)$ the Floyd metric. The claim will be shown by induction on the order, $k$, of thickness. In particular, we will prove the following stronger claim:  given $\epsilon > 0$, there exists an $N$ such that for all $x,y \in X$ with $d(x, b), d(y, b) >N$, we have that $d_f(x,y) < \epsilon$. 

The base case when $X$ is thick of order $0$ follows as a particular case of Proposition \ref{prop_div_floyd} as explained in Remark \ref{rmk_linear_div}. Note that the conclusion in the proof of that proposition is actually the stronger claim required by the induction hypothesis. We now assume the claim is true for thick spaces of order at most $k-1$, and we assume that $X$ is thick of order $k$ given by a $C$-tight network, $\mathcal{Y}$, of thick order at most $k-1$ subspaces.

Let $C' = (k+2)C$. Given any $Y \subset \mathcal{Y}$, by Lemma \ref{lemma_bi_inf_geos}, every vertex in $N_C(Y)$ is in the $C'$ neighborhood of some bi-infinite $C'$--quasi-geodesic contained in $Y$. By Lemma \ref{lemma_inf_ray}, there exists constants $K$ and $R$ such that given any $r > R$ and $x \in N_C(Y) \setminus B_b(r)$, there exists a $C'$--quasi-geodesic ray distance at most $C'$ from $x$ which does not intersect the ball $B' = B_b(\frac{r}{KC'})$. Using Karlsson's Lemma \ref{lemma_karlsson}, we choose $N > R$ so that any $C'$--quasi-geodesic which does not intersect the ball $B'$ has Floyd length less than $\frac{\epsilon}{6}$.

Let $x, y \in X \backslash B_b(N)$. As $\mathcal{Y}$ is a thick network, there exists a sequence $Y_1, Y_2, ..., Y_m$ of subspaces in $\mathcal{Y}$ such that $x \in N_C(Y_1)$, $y \in N_C(Y_m)$ and $Y_{i} \cap N_C(Y_{i+1})$ is infinite diameter for each $1 \le i < m$. By the previous paragraph, there exist infinite $C'$--quasi-geodesic rays, $\beta_1 \subset Y_1$ and $\beta_2 \subset Y_m$ based respectively at $x' \in Y_1$ and $y' \in Y_m$ such that $d(x', x), d(y',y) \le C'$. Additionally, $\beta_1$ and $\beta_2$ each do not intersect the ball $B'$.

We note that $d_f(x,x') < \frac{\epsilon}{6}$ as any geodesic between $x$ and $x'$ remains outside the ball $B'$. Similarly, $d_f(y',y)  < \frac{\epsilon}{6}$. Furthermore, given any points $p, q \in \beta_1$, we also have that $d_f(p,q) < \frac{\epsilon}{6}$. The same holds for points on $\beta_2$.

For $1 \le i \le m$, $Y_i$ is $C$--thick in the subspace metric. By the induction hypothesis we can choose $N_i$ such that given any $y_1, y_2 \in Y_i$ with $d(y_1, b), d(y_2, b) > N_i$, it follows that  $d_f(y_1,y_2) < \frac{\epsilon}{12m}$. Set $N' = \max \{N, N_1,N_2,...,N_m\}$.

Let $x_1''$ and $x_m''$ be points respectively on $\beta_1$ and $\beta_2$ such that $d(x_1'',b), d(x_m'', b) > N'$. For each $1 \le i < m$, choose points $x_i \in Y_i \cap N_C(Y_{i+1})$ such that $d(x_i, b) > N' + C$. This is possible as these sets have infinite diameter. Furthermore, choose $x_i' \in Y_{i+1}$, for $1 \le i < m$, such that $d(x_i, x_i') < C$.

We get the following bound on the Floyd distance, $d_f(x_i, x_{i+1})$:
\[d_f(x_i, x_{i+1}) \le d_f(x_i, x_i') + d_f(x_i', x_{i+1}) < 2\frac{\epsilon}{12m} = \frac{\epsilon}{6m} \]

Finally, we are ready to bound the Floyd distance, $d_f(x,y)$:
\begin{align*}
d_f(x, y) &\le d_f(x, x') + d_f(x', x'') + d_f(x'', x_1) + \Big(\sum_{i=1}^{m-1} {d_f(x_i, x_{i+1})}\Big) + d_f(x_m, y'') + d_f(y'', y') + d_f(y', y)  \\
&\le \frac{\epsilon}{6} + \frac{\epsilon}{6} + \frac{\epsilon}{12m} + (m-1)\frac{\epsilon}{6m} + \frac{\epsilon}{12m} + \frac{\epsilon}{6} + \frac{\epsilon}{6} \\
& = \frac{5 \epsilon}{6} < \epsilon
\end{align*}
\end{proof}

\bibliographystyle{amsalpha}
\bibliography{mybibliography}

\end{document}